\documentclass[12pt]{amsart}
\usepackage{hyperref}
\usepackage[shortlabels]{enumitem}
\usepackage{amsmath,amssymb,amsthm,amscd}
\newtheorem{theorem}{Theorem}
\newtheorem{lemma}[theorem]{Lemma}
\newtheorem{corollary}[theorem]{Corollary}
\newtheorem{proposition}[theorem]{Proposition}
\newtheorem*{theorem*}{Theorem}

\newtheorem{conjecture}{Conjecture}
\theoremstyle{definition}
\newtheorem{remark}{Remark}
\usepackage[margin=3cm]{geometry}
\newcommand{\Z}{\mathbb{Z}}

\newcommand{\Q}{\mathbb{Q}}

\newcommand{\Gal}{\textrm{Gal}}
\newcommand{\Aut}{\textrm{Aut}}

\newcommand{\GL}{\textrm{GL}}

\newcommand{\tors}{\textrm{tors}}
\newcommand{\cE}{\mathcal{E}}

\usepackage{mathrsfs}
\usepackage{url}
\date{}
\author{Tyler Genao}
\thanks{Email: \texttt{tylergenao@uga.edu}. 
This material is based upon work supported by the National Science Foundation Graduate Research Fellowship under Grant No. 1842396. Partial support was also provided by the Research and Training Group grant DMS-1344994 funded by the National Science Foundation.}
\title[Polynomial Bounds on Torsion From a Fixed Geometric Isogeny Class]{Polynomial Bounds on Torsion From a Fixed Geometric Isogeny Class of Elliptic Curves}
\begin{document}
\begin{abstract}
We show there exist polynomial bounds on torsion of elliptic curves which come from a fixed geometric isogeny class. More precisely, for an elliptic curve $E_0$ defined over a number field $F_0$, for each $\epsilon>0$ there exist constants $c_\epsilon:=c_\epsilon(E_0,F_0),C_\epsilon:=C_\epsilon(E_0,F_0)>0$ such that for any elliptic curve $E_{/F}$ geometrically isogenous to $E_0$, if $E(F)$ has a point of order $N$ then
\[
N\leq c_\epsilon\cdot [F:\Q]^{1/2+\epsilon},
\]
and one also has
\[
\# E(F)[\tors] \leq C_\epsilon\cdot [F:\Q]^{1+\epsilon}.
\]

\end{abstract}
\maketitle
\section{Introduction}
For an elliptic curve $E$ defined over a number field $F$, the Mordell-Weil theorem states that the abelian group $E(F)$ of $F$-rational points on $E$ is finitely generated. One consequence of this is that its torsion subgroup $E(F)[\tors]$ is finite. In fact, a celebrated result of Merel \cite[Corollaire]{Mer96} showed that the size $\#E(F)[\tors]$ is uniformly bounded in the degree of $F$; more precisely, for each integer $d\in\Z^+$ there exists a bound $B(d)$ on torsion subgroup sizes $\#E(F)[\tors]$ over all elliptic curves $E_{/F}$ where $[F:\Q]=d$.

Sharp values of $B(d)$ are only known for $d\leq 3$, via a complete classification of torsion groups of elliptic curves \cite{Maz77, KM88, Kam92a, Kam92b, DEvHMZB}. However, Merel \cite{Mer96} gave an explicit upper bound on prime power divisors of $\#E(F)[\tors]$ in terms of $d$, which was later strengthened by Parent \cite[Corollaire 1.8]{Par99}: if $p^n\mid \#E(F)[\tors]$ then $p^n\leq 129(5^d-1)(3d)^6$. This gives bounds $B(d)$ which are larger than exponential in the degree $d$.

It is a folklore conjecture that there exist polynomial bounds on torsion groups of elliptic curves over number fields. More precisely:
\begin{conjecture}\cite{CCS13}\label{ConjPolyBds}
There exist constants $C,\alpha>0$ such that for all elliptic curves $E_{/F}$ one has $\#E(F)[\emph{tors}]\leq C\cdot [F:\Q]^\alpha$.
\end{conjecture}
Parent's bounds above \cite{Par99} are more than an exponential factor away from this conjecture. However, several results in the literature support this conjecture once we restrict certain parameters of our elliptic curves. For example, for any elliptic curve $E_{/F}$ with integral $j$-invariant, Hindry and Silverman have shown that $\#E(F)[\tors]\leq 1977408\cdot d\log(d)$ when $d:=[F:\Q]>1$ \cite[Th\'{e}or\`{e}me 1]{HS99}. In a stricter case, if $E$ has complex multiplication (CM) then Clark and Pollack have shown that $\#E(F)[\tors] \leq C \cdot d\log\log d$ when $d>2$, where $C\in\Z^+$ is some absolute, effectively computable constant \cite[Theorem 1]{CP15}. 

There are also polynomial bounds for elliptic curves with rational $j$-invariant:
Clark and Pollack have shown that for each $\epsilon>0$ there exists a constant $C_\epsilon>0$ such that for any elliptic curve $E_{/F}$ whose $j$-invariant $j(E)\in \Q$, one has that the exponent\footnote{Given a finite group $(G,+)$, its \textit{exponent} $\exp G$ is the least integer $n\in\Z^+$ such that $nG=0$. When $G$ is abelian, its exponent is equal to the largest possible order of an element in $G$.}  $\exp E(F)[\tors] \leq C_\epsilon\cdot [F:\Q]^{3/2+\epsilon}$, and thus $\#E(F)[\tors]\leq C_\epsilon\cdot [F
:\Q]^{5/2+\epsilon}$ \cite[Theorem 1.3]{CP18}.
An identical result holds when one assumes the Generalized Riemann Hypothesis (GRH) and replaces $\Q$ with a number field $F_0$ which contains no Hilbert class field of any imaginary quadratic field 
\cite[Theorem 1.6]{CP18}.

The principal result of this paper constructs polynomial bounds on orders of torsion points (and thus torsion groups) of non-CM elliptic curves $E_{/F}$ within a fixed geometric isogeny class. Recall that an \textit{isogeny} between elliptic curves $E$ and $E'$ is a nonconstant algebraic map $\phi\colon E\rightarrow E'$ which preserves basepoints. We say that $\phi$ is $F$-rational if $E$, $E'$ and $\phi$ are defined over $F$. As an adjective, ``geometric" will mean $\overline{\Q}$-rational, where $\overline{\Q}$ is a fixed algebraic closure of $\Q$.
\begin{theorem}\label{ThmPolyBoundsIsogenyClass}
Fix a number field $F_0$ and a non-CM elliptic curve ${E_0}_{/F_0}$. Then for each $\epsilon>0$ there exist constants $c_\epsilon:=c_\epsilon(E_0,F_0),C_\epsilon:=C_\epsilon(E_0,F_0)>0$ such that for any elliptic curve $E_{/F}$ geometrically isogenous to ${E_0}_{/F_0}$, one has both
\[
\exp E(F)[\emph{tors}]\leq c_\epsilon\cdot [F:\Q]^{1/2+\epsilon}
\] 
and
\[
\#E(F)[\emph{tors}]\leq C_\epsilon\cdot [F:\Q]^{1+\epsilon}.
\]
\end{theorem}
\begin{remark}
In Theorem \ref{ThmPolyBoundsIsogenyClass}, The power ``$1/2+\epsilon$" in the exponent bound $c_\epsilon\cdot [F:\Q]^{1/2+\epsilon}$ is optimal ``up to $\epsilon$", since for any elliptic curve $E_{/F}$, any integer $N\in\Z^+$ and any torsion point $R\in E[N]$, one always has $[F(R):F]\leq N^2$.
\end{remark}
\begin{remark}\label{Rmk_CMcase}
We must assume in Theorem \ref{ThmPolyBoundsIsogenyClass} that our elliptic curves are non-CM to have the \textit{torsion group exponent} bound $c_\epsilon\cdot [F:\Q]^{1/2+\epsilon}$ hold. Indeed, one can show using e.g. \cite[Corollary 1.8]{BC20} that for any imaginary quadratic field $K$, the geometric isogeny class of elliptic curves with CM field $K$ contains infinitely many elliptic curves $E_{/F}$ with $[F:\Q]$ arbitrarily large and $\exp E(F)[\tors]>[F:\Q]$. Despite this, as noted earlier there is an asymptotically sharp bound on the size of \textit{full torsion groups} of CM elliptic curves: one always has $\# E(F)[\tors]\leq C\cdot [F:\Q]\log\log [F:\Q]$ when $[F:\Q]>2$ for some absolute, effectively computable constant $C\in\Z^+$ \cite[Theorem 1]{CP15}.
\end{remark}

With Remark \ref{Rmk_CMcase} in mind, we will assume for the rest of this paper that our elliptic curves have no geometric CM. 
A key step for us in polynomially bounding torsion from a non-CM geometric isogeny class $\cE$ will be to relate the adelic indices of two rationally isogenous non-CM elliptic curves (this is Corollary \ref{CorollaryAdelicDivisibility}).

In contrast to \cite[Theorem 1.3]{CP18}, the collection of elliptic curves in Theorem \ref{ThmPolyBoundsIsogenyClass} will contain curves whose $j$-invariants $j'$ have arbitrarily large degrees $[\Q(j'):\Q]$. However, both Theorem \ref{ThmPolyBoundsIsogenyClass} and \cite[Theorem 1.3]{CP18} are part of a natural uniformity conjecture on torsion groups that is motivated by our current understanding of Galois representations of rational elliptic curves. 
\begin{conjecture}\label{Conj_PolyBdsForI_Q}
There exist constants $C, \alpha>0$ such that for all elliptic curves $E_{/F}$ geometrically isogenous to some elliptic curve defined over $\Q$, one has $\# E(F)[\emph{tors}] \leq C\cdot [F:\Q]^{\alpha}$.
\end{conjecture}
This is a special case of Conjecture \ref{ConjPolyBds}. There is recent work which suggests its tractability: a result of Bourdon and Najman \cite[Proposition 4.1]{BN} can be used to show that when $[F:\Q]$ is odd and $E_{/F}$ is $\overline{\Q}$-isogenous to a rational elliptic curve, one has $\exp E(F)\leq 720720\sqrt{35}\cdot [F:\Q]^{1/2}$, and thus $\# E(F)[\tors] \leq 1441440\sqrt{35}\cdot [F:\Q]^{1/2}$. On the other hand, if one assumes a uniformity conjecture of Zywina on indices of adelic Galois representations of non-CM elliptic curves over $\Q$ \cite[Conjecture 1.3]{Zyw}, then Conjecture \ref{Conj_PolyBdsForI_Q} follows ``up to $\epsilon$" with the same bounds as in Theorem \ref{ThmPolyBoundsIsogenyClass}; the principal difference is that the constants in Theorem \ref{ThmPolyBoundsIsogenyClass} will change and depend only on $\epsilon$.

\subsection{Acknowledgments}
The author thanks Pete L. Clark for his comments on an earlier draft of this paper, and the suggestion that Greenberg's proof of \cite[Proposition 2.1.1]{Gre12} might be adjustable to give a stronger result, which is now Proposition \ref{PropGreenbergNadic}.
The author also thanks the referee for their insightful comments, particularly on improving the degree bounds in Theorem \ref{ThmPolyBoundsIsogenyClass}. Finally, the author thanks Jacob Mayle for his comment that the original version of Corollary \ref{CorollaryAdelicDivisibility}, which was an equality of $n$-adic indices, implied an equivalence of adelic indices; this simplifies some of the presentation of this paper.
\section{Results on Galois Representations of Non-CM Elliptic Curves}
\subsection{Some profinite group theory}
In this section, we will show that a result of Greenberg on $\ell$-adic Galois representations \cite[Proposition 2.1.1]{Gre12} has a proof which applies to $n$-adic representations for composite $n\in\Z^+$, after some modifications. We will then use this composite version to prove that rationally isogenous non-CM elliptic curves have adelic Galois representations with equal indices in $\GL_2(\hat{\Z})$, a fact we will use in our proof of Theorem \ref{ThmPolyBoundsIsogenyClass}; this is recorded as Corollary \ref{CorollaryAdelicDivisibility}.

Before we prove this adelic index result, we will prove a few general facts about subgroups of $\GL_2(\hat{\Z})$. For each integer $n\in\Z^+$, we will denote by $\pi_n\colon \GL_2(\hat{\Z})\twoheadrightarrow \GL_2(\Z/n\Z)$ the mod-$n$ reduction map, and by $\pi_{n^\infty}\colon \GL_2(\hat{\Z})\twoheadrightarrow \GL_2(\Z_n)$ the $n$-adic reduction map. 

By profinite group theory, for any subgroup $G\subseteq \GL_2(\hat{\Z})$ one has that $G$ is open in $\GL_2(\hat{\Z})$ iff $G$ has finite index in $\GL_2(\hat{\Z})$, iff $G$ contains contains an open neighborhood $U(M):=\ker\pi_M$ for some $M\in\Z^+$. When $G$ is open, we will call the least such $M$ for which $U(M)\subseteq G$ the \textit{level of $G$.} 
\begin{lemma}\label{Lemma_AllEqualNAdicIndicesImpliesEqualAdelicIndex}
Let $G$ be a subgroup of $\emph{GL}_2(\hat{\Z})$. 
\begin{enumerate}[a.]
\item One has for all $n\in\Z^+$ that $U(n)\subseteq G$ iff $G=\pi_n^{-1}(\pi_n(G))$.
\item If $U(n)\subseteq G$ then
\[
[\emph{GL}_2(\hat{\Z}):G]=[\emph{GL}_2(\Z_n):\pi_{n^\infty}(G)]=[\emph{GL}_2(\Z/n\Z):\pi_n(G)].
\] 
\item Suppose that $G$ is open, and let $G'\subseteq \emph{GL}_2(\hat{\Z})$ be another open subgroup. If for all $n\in\Z^+$ one has
\[
[\emph{GL}_2(\Z_n):\pi_{n^\infty}(G)]=[\emph{GL}_2(\Z_n):\pi_{n^\infty}(G')]
\]
then one has the equality
\[
[\emph{GL}_2(\hat{\Z}):G]=[\emph{GL}_2(\hat{\Z}):G'].
\]
\end{enumerate}
\end{lemma}
\begin{proof}
For part a., suppose first that $U(n)\subseteq G$. To check that $G=\pi_n^{-1}(\pi_n(G))$, we note that the containment $\subseteq$ is clear. For the reverse containment, observe that if $x\in \pi_n^{-1}(\pi_n(G))$ then $\pi_n(x)\in \pi_n(G)$, and so $\pi_n(x)=\pi_n(g)$ for some $g\in G$; in particular, $xg^{-1}\in \ker\pi_n\subseteq G$, whence we have $x\in G$.
For the converse, assume that $\pi_n^{-1}(\pi_n(G))=G$. Then we have $U(n)=\pi_n^{-1}(\lbrace I\rbrace)\subseteq G$, where $I\in \GL_2(\Z/n\Z)$ is the identity matrix.

Part b. follows from the general fact that if $f\colon G_0\rightarrow K$ is a group homomorphism and $G\subseteq G_0$ is a subgroup containing the kernel $\ker f$, then a set of coset representatives $\lbrace f(g_i)\rbrace_i$ for $f(G)$ in $f(G_0)$ lifts to a set of coset representatives $\lbrace g_i\rbrace_i$ for $G$ in $G_0$.
In particular, when $G\subseteq \GL_2(\hat{\Z})$ is a subgroup with $U(n)\subseteq G$, one has that $[\GL_2(\hat{\Z}):G]=[\GL_2(\Z/n\Z):\pi_{n}(G)]$. It also follows that $[\GL_2(\hat{\Z}):G]=[\GL_2(\Z_n):\pi_{n^\infty}(G)]$, via the containment $\ker\pi_{n^\infty}\subseteq G$ (the map $\pi_{n}$ factors through $\pi_{n^\infty}$).

For part c., let us set $N:=\textrm{lcm}(M,M')$ where $M$ and $M'$ are the levels of $G$ and $G'$ respectively. Since $U(N)\subseteq U(M)\subseteq G$ and $U(N)\subseteq U(M')\subseteq G'$, by part b. we have both
\[
[\GL_2(\hat{\Z}):G]=[\GL_2(\Z_N):\pi_{N^\infty}(G)]
\]
and
\[
[\GL_2(\hat{\Z}):G']=[\GL_2(\Z_N):\pi_{N^\infty}(G')].
\]
Thus, our hypothesis implies that $[\GL_2(\hat{\Z}):G]=[\GL_2(\hat{\Z}):G']$.
\end{proof}
\subsection{A composite version of a result of Greenberg}
Our next goal is to prove a composite version of \cite[Proposition 2.1.1]{Gre12}. Given an integer $n\in\Z^+$, let us recall that the \textit{ring of $n$-adic integers} is
\[
\Z_n:=\varprojlim_{\ell\mid n, k\geq 1}\Z/\ell^k\Z\cong \prod_{\ell\mid n}\Z_\ell.
\]
Following this, the \textit{ring of $n$-adic numbers} is
\[
\Q_n:=\prod_{\ell\mid n}\Q_\ell.
\]
For a free $\Q_n$-module $V$ of finite rank, we call the $\Z_n$-span of any basis of $V$ a \textit{$\Z_n$-lattice. }

\begin{proposition}\label{PropGreenbergNadic}
Fix a positive integer $n$. Let $V$ be a free finite rank $\Q_n$-module. Suppose that $G$ is a compact open subgroup of $\emph{Aut}_{\Q_n}(V)$. If $T$ and $T'$ are two $G$-invariant $\Z_n$-lattices in $V$, then 
\[
[\emph{Aut}_{\Z_n}(T):G]=[\emph{Aut}_{\Z_n}(T'):G].
\]
\end{proposition}
\begin{proof}
Suppose that $V$ is free of rank $d$ over $\Q_n$. Fixing a basis for $V$, one has an isomorphism $\Aut_{\Q_n}(V)\cong \GL_d(\Q_n)\cong \prod_{\ell\mid n}\GL_d(\Q_\ell)$. 

For each prime $\ell\in\Z^+$, the group $\GL_d(\Q_\ell)$ is a locally compact topological group, and thus has a left Haar measure. In fact, since $\GL_d(\Q_\ell)$ is a reductive $\ell$-adic group it is also \textit{unimodular}: every left Haar measure is also a right Haar measure \cite[Theorem 5.1]{Glo96}. It follows then that the finite product $\prod_{\ell\mid n} \GL_d(\Q_\ell)\cong \GL_d(\Q_n)$ is also unimodular for composite $n\in\Z^+$.

Fix a Haar measure $\mu$ on $\GL_d(\Q_n)$; since $G$ is compact open in $\GL_d(\Q_n)$, we have $\mu(G)>0$, so we may assume that $\mu(G)=1$.
Given a $\Z_n$-lattice $T$ in $V$, we can identify $\Aut_{\Z_n}(T)\cong \GL_d(\Z_n)$ once we choose a $\Z_n$-basis for $T$. For any $\sigma\in \Aut_{\Q_n}(V)$ one has that $\sigma(T)$ is a $\Z_n$-lattice; this gives us an action of $\Aut_{\Q_n}(V)$ on the set of $\Z_n$-lattices in $V$. This action is clearly transitive, and the stabilizer of any $\Z_n$-lattice $T$ is $\Aut_{\Z_n}(T)$. Additionally, $\Aut_{\Z_n}(T)$ is a compact open subgroup of $\Aut_{\Q_n}(V)$, and $G$ is contained in $\Aut_{\Z_n}(T)$ and has finite index. Since $\Aut_{\Z_n}(T)$ is a finite disjoint union of left cosets of $G$, and since $\mu(G)=1$ and $\mu$ is left invariant, it follows that 
\begin{equation}\label{EqnMeasureEqualsIndex}
\mu(\Aut_{\Z_n}(T))=[\Aut_{\Z_n}(T):G].
\end{equation}

Let $T$ and $T'$ be $G$-invariant $\Z_n$-lattices of $V$.  
Since $\Aut_{\Q_n}(V)$ acts transitively on $\Z_n$-lattices, there exists $\sigma\in \Aut_{\Q_n}(V)$ with $\sigma(T)=T'$. It follows then that $\Aut_{\Z_n}(T')=\sigma \Aut_{\Z_n}(T)\sigma^{-1}$. As $\mu$ is both left and right invariant, we conclude that $\mu(\Aut_{\Z_n}(T'))=\mu(\Aut_{\Z_n}(T))$, which by \eqref{EqnMeasureEqualsIndex} implies our result.
\end{proof}
\subsection{Galois representations of elliptic curves}
Given an elliptic curve $E$ over a number field $F$, for each integer $n\in\Z^+$ the absolute Galois group $G_F:=\Gal(\overline{F}/F)$ acts on the $n$-torsion subgroup $E[n]$ of $E$. This action is described by the \textit{mod-$n$ Galois representation of $E$,} denoted by
\[
\rho_{E,n}\colon G_F\rightarrow \Aut(E[n]).
\]
Since $E[n]$ is a free $\Z/n\Z$-module of rank two, choosing a basis $\lbrace P,Q\rbrace$ for $E[n]$ gives an isomorphism $\Aut(E[n])\cong\GL_2(\Z/n\Z)$; we will often work with a basis implicitly, suppressing dependence on one. 

The action of $G_F$ on each torsion subgroup $E[n]$ for all $n\in\Z^+$ induces an action on their inverse limit $T(E):=\varprojlim E[n]$, called the \textit{adelic Tate module of $E_{/F}$}. Since each $E[n]$ is a free rank two $\Z/n\Z$-module, it follows that $T(E)$ is free of rank two over the profinite integers $\hat{\Z}:=\varprojlim \Z/n\Z$. The action of $G_F$ on $T(E)$ is called the \textit{adelic Galois representation of $E_{/F}$,} denoted by
\[
\rho_E\colon G_F\rightarrow \Aut_{\hat{\Z}}(T(E)).
\]
This also describes the action of $G_F$ on the full torsion subgroup $E[\tors]$.
Choosing a basis for $T(E)$ gives an isomorphism $\Aut_{\hat{\Z}}(T(E))\cong \GL_2(\hat{\Z})$. Assume hereafter that our elliptic curves are non-CM; then it follows by \cite[Th\'eor\`eme 2]{Ser72} that the image $\rho_E(G_F)$ is open in $\GL_2(\hat{\Z})$. We say that the \textit{adelic level of $E_{/F}$} is the level of $\rho_{E}(G_F)$ as a subgroup of $\GL_2(\hat{\Z})$.
By abuse of notation, we will often suppress its dependence on $E$ and $F$.

Given an integer $n\in\Z^+$, let us define the \textit{$n$-adic Tate module of $E_{/F}$} as $T_n(E):=\varprojlim_{k\geq 1}E[n^k]$. We have that $T_n(E)$ is a free rank two $\Z_n$-module. 
The \textit{$n$-adic representation of $E_{/F}$} is the action of $G_F$ on $T_n(E)$, denoted by
\[
\rho_{E,n^\infty}\colon G_F\rightarrow \Aut_{\Z_n}(T_n(E)).
\]
This also describes the action of $G_F$ on the $n$-primary torsion subgroup $E[n^\infty]:=\bigcup_{k\geq 1}E[n^k]=\sum_{\ell\mid n} E[\ell^\infty]$. 
Since $\rho_{E,n^\infty}(G_F)$ is a projection of $\rho_E(G_F)$, it is open in $\GL_2(\Z_n)$.

The action of $G_F$ on $T_n(E)$ extends naturally to an action on the rational $n$-adic Tate module $V_n(E):=T_n(E)\otimes_{\Z_n} \Q_n$. We can realize $\rho_{E,n^\infty}(G_F)$ as finite-index subgroup of $\Aut_{\Z_n}(T_n(E))$, the latter of which is a compact open subgroup of $\Aut_{\Q_n}(V_n(E))$. 

Suppose two elliptic curves $E_{/F}$ and $E'_{/F}$ are $F$-isogenous; let us write this isogeny as $\phi\colon E\rightarrow E'$. Choose an integer $n\in\Z^+$; then this isogeny induces a $\Z_n[G_F]$-module homomorphism $\phi\colon T_n(E)\rightarrow T_n(E')$. In fact, we have a short exact sequence of $\Z_n[G_F]$-modules,
\[
0\rightarrow T_n(E)\xrightarrow{\phi} T_n(E')\rightarrow C\rightarrow 0,
\]
for some finite module $C$.
Tensoring this sequence to $\Q_n$ shows that the rational Tate modules $V_n(E)$ and $V_n(E')$ are isomorphic $G_F$-modules, and so $T_n(E)$ and $T_n(E')$ may be realized as $G_F$-stable $\Z_n$-lattices in $V_n(E)$. By Proposition \ref{PropGreenbergNadic}, this implies that
\[
[\GL_2(\Z_n):\rho_{E,n^\infty}(G_F)]=[\GL_2(\Z_n):\rho_{E',n^\infty}(G_F)].
\]
Since $n\in\Z^+$ was arbitrary, we have proven the following key result after applying Lemma \ref{Lemma_AllEqualNAdicIndicesImpliesEqualAdelicIndex}.
\begin{corollary}\label{CorollaryAdelicDivisibility}
Let $E_{/F}$ and $E'_{/F}$ be $F$-isogenous non-CM elliptic curves. Then one has
\[
[\emph{GL}_2(\hat{\Z}):\rho_{E}(G_F)]=[\emph{GL}_2(\hat{\Z}):\rho_{E'}(G_F)].
\]
\end{corollary}
Let us note one more fact about Galois representations of elliptic curves with a rational torsion point. 
For each integer $n\geq 2$, we define a distinguished subgroup of $\GL_2(\Z/n\Z)$,
\[
B_1(n):=\left\lbrace \begin{bmatrix}
1&b\\
0&d
\end{bmatrix}\in \GL_2(\Z/n\Z)\right\rbrace.
\]
When an elliptic curve $E_{/F}$ has an $F$-rational order $n$ torsion point, it follows that the image $\rho_{E,n}(G_F)$ is contained in $B_1(n)$ up to conjugacy. This implies the index divisibility 
\[
[\GL_2(\Z/n\Z):B_1(n)]\mid [\GL_2(\Z/n\Z):\rho_{E,n}(G_F)].
\]
The former index can be written more explicitly. Let us recall Euler's phi function $\varphi\colon \Z^+\rightarrow\Z^+$ and the Dedekind psi function $\psi\colon \Z^+\rightarrow \Z^+$, both arithmetic multiplicative functions defined on prime powers via $\varphi(\ell^k)=\ell^{k-1}(\ell-1)$ and $\psi(\ell^k)=\ell^{k-1}(\ell+1)$ respectively.
\begin{lemma}\label{LemmaIndexOfB1N}
For $n\geq 2$ one has
\[
[\emph{GL}_2(\Z/n\Z):B_1(n)]=\varphi(n)\psi(n).
\]
\end{lemma}
\begin{proof}
See e.g. \cite[\S 7.2]{CGPS22}.
\end{proof}
\section{Polynomial Bounds on Torsion}
We are ready to prove the main result of this paper.
\begin{proof}[Proof of Theorem \ref{ThmPolyBoundsIsogenyClass}]
To recapitulate Theorem \ref{ThmPolyBoundsIsogenyClass}, we will show that for any fixed non-CM elliptic curve ${E_0}_{/F_0}$, for all $\epsilon>0$ there exist constants $c_\epsilon:=c_\epsilon(E_0,F_0),C_\epsilon:=C_\epsilon(E_0,F_0)>0$ such that for all elliptic curves $E_{/F}$ geometrically isogenous to $E_0$, one has both 
\[
\exp E(F)[\tors]\leq c_\epsilon\cdot [F:\Q]^{1/2+\epsilon}
\]
and
\[
\# E(F)[\tors]\leq C_\epsilon\cdot [F:\Q]^{1+\epsilon}.
\]
First, observe that the desired bound on $\# E(F)[\tors]$ will follow from the desired bound on the exponent $\exp E(F)[\tors]$, via the divisibility
\[
\# E(F)[\tors]\mid (\exp E(F)[\tors])^2
\]
(one can take $C_\epsilon:=c_{\epsilon/2}^2$).
To this end, our proof will focus on bounding $\exp E(F)[\tors]$. 

Let us write $n:=\exp E(F)[\tors]$. Then up to conjugacy we have $\rho_{E,n}(G_F)\subseteq B_1(n)$, so by Lemma \ref{LemmaIndexOfB1N} we get
\begin{equation}\label{EqnOrderM^aTorsionPointDegree}
\varphi(n)\psi(n)\mid [\GL_2(\Z/n\Z):\rho_{E,n}(G_F)].
\end{equation}
By \cite[Lemma 3.1]{LFN20} there exists a(n at worst) quadratic extension $L/FF_0$ for which $E$ and $E_0$ are $L$-isogenous. 
Thus, Corollary \ref{CorollaryAdelicDivisibility} implies that
\begin{equation}\label{Eqn_EquivalentAdelicIndicesOverQuadratic}
[\GL_2(\hat{\Z}):\rho_{E}(G_{L})]=[\GL_2(\hat{\Z}):\rho_{E_0}(G_{L})].
\end{equation}
Since the extension $F_0(E_0[\tors])/F_0$ is normal, so is the extension $L(E_0[\tors])/L$, and we have
\[
\Gal(L(E_0[\tors])/L)\cong \Gal(F_0(E_0[\tors])/L\cap F_0(E_0[\tors]))
\]
(this general fact is e.g. \cite[Proposition 7.15]{Mil}). Since $\rho_{E_0}(G_L)\cong \Gal(L(E_0[\tors])/L)$, we see that $\rho_{E_0}(G_L)$ is a subgroup of $\rho_{E_0}(G_{F_0})$ of index $[L\cap F_0(E_0[\tors]):F_0]$. Thus, we deduce that
\begin{equation}\label{Eqn_AdelicIndexBaseChange}
[\GL_2(\hat{\Z}):\rho_{E_0}(G_{L})]\mid [\GL_2(\hat{\Z}):\rho_{E_0}(G_{F_0})]\cdot [L:F_0].
\end{equation}
Finally, since $[FF_0:\Q]=[FF_0:F]\cdot [F:\Q]\mid [F_0:\Q]!\cdot [F:\Q]$ and $[L:F_0]\mid 2[FF_0:F_0]$, we find that 
\[
[L:F_0]\mid 2([F_0:\Q]-1)!\cdot [F:\Q].
\]
Combining this fact with \eqref{EqnOrderM^aTorsionPointDegree}, \eqref{Eqn_EquivalentAdelicIndicesOverQuadratic} and \eqref{Eqn_AdelicIndexBaseChange}, we conclude that
\begin{equation}\label{Eqn_DegreeDivByN}
\varphi(n)\psi(n)\mid 2I([F_0:\Q]-1)!\cdot [F:\Q],
\end{equation}
where $I:=[\GL_2(\hat{\Z}):\rho_{E_0}(G_{F_0})]$ is the adelic index of our fixed elliptic curve ${E_0}_{/F_0}$.

One can check directly that $\psi(n)>n$ for any $n>1$. 
Fixing an $\epsilon\in (0,1)$, by \cite[Theorem 327]{HW08} there exists a constant $b_\epsilon>0$ such that for all $n\in\Z^+$ one has 
\[
\varphi(n)>b_\epsilon\cdot n^{1-\epsilon}.
\]
Thus, from \eqref{Eqn_DegreeDivByN} we deduce that
\[
n^{2-\epsilon}< 2Ib_\epsilon^{-1}([F_0:\Q]-1)!\cdot [F:\Q].
\]
Since $n:=\exp E(F)[\tors]$, we conclude that
\[
\exp E(F)[\tors]< c_\epsilon\cdot [F:\Q]^{1/2+\epsilon}
\]
where $c_\epsilon:=c_\epsilon(E_0,F_0):=(2Ib_\epsilon^{-1}([F_0:\Q]-1)!)^{1/(2-\epsilon)}$,
which is the desired upper bound on the exponent of $E(F)[\tors]$.
\end{proof}


\begin{thebibliography}{GMGS1111}

\bibitem[BC20]{BC20}
A. Bourdon and P.L. Clark,
\emph{Torsion points and {G}alois representations on {CM} elliptic
              curves},
              Pacific J. Math. 305 (2020), 43--88.
              
\bibitem[BN]{BN} A. Bourdon and F. Najman,
\emph{Sporadic points of odd degree on $X_1(N)$ coming from $\Q$-curves,}
preprint, \url{https://arxiv.org/abs/2107.10909}.

\bibitem[CCS13]{CCS13}
P.L. Clark, B. Cook and J. Stankewicz,
\emph{Torsion points on elliptic curves with complex multiplication
              (with an appendix by {A}lex {R}ice)},
Int. J. Number Theory 9 (2013), 447-479.

\bibitem[CGPS22]{CGPS22}
P.L. Clark, T. Genao, P. Pollack and F. Saia,
\emph{The least degree of a {CM} point on a modular curve},
J. Lond. Math. Soc. (2) 105 (2022), no. 2,
825--883.

\bibitem[CP15]{CP15}
P.L. Clark and P. Pollack,
\emph{The truth about torsion in the CM case},
C. R. Math. Acad. Sci. Paris 353 (2015), no. 8, 683--688.

\bibitem[CP18]{CP18}
P.L. Clark and P. Pollack,
\emph{Pursuing polynomial bounds on torsion},
Israel J. Math 227 (2018),
889--909.

\bibitem[CN21]{CN21} J. Cremona and F. Najman, 
\emph{$\Q$-curves over odd degree number fields,} 
Res. Number Theory 7 (2021),
Paper No. 62, 30.

\bibitem[DEvH+21]{DEvHMZB}
M. Derickx, A. Etropolski, M. van Hoeij, J. Morrow and D. Zureick-Brown,
\emph{Sporadic cubic torsion}, Algebra Number Theory 15 (2021),
1837--1864.

\bibitem[Gl\"o96]{Glo96}
H. Gl\"ockner,
\emph{Haar measure on linear groups over local skew fields},
J. Lie Theory 6 (1996),
165--177.

\bibitem[Gre12]{Gre12}
R. Greenberg,
\emph{The image of {G}alois representations attached to elliptic curves with an isogeny}, 
Amer. J. Math. 134 (2012),
1167--1196.   

\bibitem[HW08]{HW08}
G.H. Hardy and E.M. Wright,
\emph{An introduction to the theory of numbers},
6th Ed., 
Oxford University Press,
Oxford (2008).

\bibitem[HS99]{HS99}
M. Hindry and J. Silverman,
\emph{Sur le nombre de points de torsion rationnels sur une courbe elliptique}, 
C. R. Acad. Sci. Paris Sér. I Math 329 (1999), no. 2, 97--100.

\bibitem[Kam92a]{Kam92a} S. Kamienny,
\emph{Torsion points on elliptic curves and {$q$}-coefficients of modular forms},
Invent. Math. 109 (1992), 221--229.

\bibitem[Kam92b]{Kam92b} 
S. Kamienny,
\emph{Torsion points on elliptic curves over fields of higher degree},
Internat. Math. Res. Notices (1992), 129--133.

\bibitem[KM88]{KM88}
M. A. Kenku and F. Momose, 
\emph{Torsion points on elliptic curves defined over quadratic fields}, Nagoya Math. J. 109 (1988), 125--149.

\bibitem[LFN20]{LFN20}
S. Le Fourn and F. Najman,
\emph{Torsion of {$\Bbb Q$}-curves over quadratic fields},
Math. Res. Lett. 27 (2020),
209--225.

\bibitem[Maz77]{Maz77} B. Mazur, \emph{Modular curves and the Eisenstein ideal,} Inst. Hautes \'{E}tudes Sci. Publ. Math. No. 47 (1977).

\bibitem[Mer96]{Mer96}
L. Merel, \emph{Bornes pour la torsion des courbes elliptiques sur les corps de nombres},
Invent. Math. 124 (1996), 437--449.        
              
\bibitem[Mil]{Mil}
J.S. Milne,
\emph{Fields and Galois theory}, v5.10, 
course notes, \url{https://www.jmilne.org/math/CourseNotes/FT.pdf}. 
              
\bibitem[Par99]{Par99} 
P. Parent,
\emph{Bornes effectives pour la torsion des courbes elliptiques sur les corps de nombres},
J. Reine Angew. Math. 506 (1999),
85--116.               

\bibitem[Ser72]{Ser72}
J.-P. Serre, 
\emph{Propri\'{e}t\'{e}s galoisiennes des points d'ordre fini des courbes elliptiques},
Invent. Math. 15 (1972), 259--331.

\bibitem[Zyw]{Zyw}
D. Zywina,
\emph{Explicit open images for elliptic curves over $\Q$},
preprint,
\url{https://arxiv.org/abs/2206.14959}.
\end{thebibliography}
\end{document}